\documentclass[a4paper,reqno]{amsart}
 \usepackage[utf8]{inputenc}
 \usepackage{a4wide}
 \usepackage{amssymb,exscale,bbm,mathrsfs,stmaryrd,xargs, units}
 \usepackage[mathscr]{euscript}
 \usepackage{color}
 \usepackage[symbol]{footmisc} 

 \newtheorem{theorem}{Theorem}[section]
 \newtheorem{corollary}[theorem]{Corollary}
 \newtheorem{lemma}[theorem]{Lemma}
 
 \theoremstyle{definition}
 \newtheorem{definition}[theorem]{Definition}
 \theoremstyle{remark}

 \numberwithin{equation}{section}

\begin{document}

%
\setlength{\parindent}{0cm}
\renewcommand{\labelenumi} {(\alph{enumi})}    
\renewcommand{\labelenumii}{(\roman{enumii})}
\renewcommand{\theenumi} {(\alph{enumi})}      
\renewcommand{\theenumii}{(\roman{enumii})}    
  \renewcommand {\Re}{\text{Re}}
  \renewcommand {\Im}{\text{Im}}
  \newcommand   {\upi}{{\mathrm \pi}}
  \newcommand   {\ud} {{\mathrm d}}
  \newcommand   {\dr}{{\mathrm d}r}
  \newcommand   {\ds}{{\mathrm d}s}
  \newcommand   {\dt}{{\mathrm d}t}
  \newcommand   {\dx}{{\mathrm d}x}
  \newcommand   {\dxi}{{\mathrm d}\xi}
  \newcommand   {\dy}{{\mathrm d}y}
  \newcommand   {\dz}{{\mathrm d}z}
  \newcommand*{\longhookrightarrow}{\ensuremath{\lhook\joinrel\relbar\joinrel\relbar\joinrel\rightarrow}}

  \newcommand {\pihalbe}{\nicefrac{\upi}{2}}
  \newcommand {\einhalb}{{\nicefrac{1}{2}}}
  \newcommand {\CC}{\mathbb C}  
  \newcommand {\RR}{\mathbb R}
  \newcommand {\KK}{\mathbb K}  
  \newcommand {\NN}{\mathbb N}
  \newcommand {\ZZ}{\mathbb Z}

  \newcommand {\BOUNDED}{\mathcal B}
  \newcommand {\DOMAIN}{\mathscr D}
  \newcommand {\eins} {\mathbbm 1}
  \newcommand {\al}{\alpha}
  \newcommand {\la}{\lambda}
  \newcommand {\eps}{\varepsilon}
  \newcommand {\Ga}{\Gamma}
  \newcommand {\ga}{\gamma}
  \newcommand {\om}{\omega}
  \newcommand {\Om}{\Omega}
  \newcommand {\norm}[1] {\| #1 \|}  
  \newcommand {\abs}[1] {|#1|}
  \newcommand {\biggabs}[1] {\bigg|#1\bigg|}
  \newcommand {\lrnorm}[1]{\left\| #1 \right\|}
  \newcommand {\bignorm}[1]{\bigl\| #1 \bigr\|}
  \newcommand {\Bignorm}[1]{\Bigl\| #1 \Bigr\|}
  \newcommand {\Biggnorm}[1]{\Biggl\| #1 \Biggr\|}
  \newcommand {\biggnorm}[1]{\biggl\| #1 \biggr\|}
  \newcommand {\Bigidual}[3] {\Bigl\langle #1, #2 \Bigr\rangle_{#3}}
  \newcommand {\bigidual}[3] {\bigl\langle #1, #2 \bigr\rangle_{#3}}
  \newcommand {\idual}[3] {\langle #1, #2 \rangle_{#3}}
  \newcommand {\Bigdual}[2] {\Bigidual{#1}{#2}{}}
  \newcommand {\bigdual}[2] {\bigidual{#1}{#2}{}}
  \newcommand {\dual}[2] {\idual{#1}{#2}{} }
  \newcommand {\sector}[1] {S_{#1}}
  \newcommand{\fra}{\mathfrak a}

  \newcommand {\SUCHTHAT}{:\;}
  \newcommand {\embeds} {\hookrightarrow}

  \newcommand {\ui}{\text{\rm i}}  
  \newcommand {\SCHWARTZ} {\mathscr S}
  \newcommand {\FOURIER} {\mathscr F}
  \newcommand {\dist}{\text{dist}}
  \newcommand {\SECTOR}[1]{S_{#1}} 
  \newcommand {\calA}{{\mathcal A}}
  \newcommand {\calL}{{\mathcal L}}
  \newcommand {\calQ}{{\mathcal Q}}
  \newcommand {\calS}{{\mathcal S}}
  \def\aDeForme{\mathop{\mbox{{\LARGE $\mathfrak{a}$}}}} 
  \newcommandx*\form[3][2=\cdot,3=\cdot]{ \aDeForme(#1; #2, #3)}
  \newcommandx*\sg[3][2=A,3=t]{e^{-{#1}\cdot {#2}(#3)}}
  \newcommand {\sprod}[2] {\left[#1\,|\,#2\right]_H}
  \newcommand {\dprod}[2] {\left\langle #1, #2 \right\rangle}
  \newcommand {\bracket}[1]{\langle{#1}\rangle}
   \newcommand {\supp}{\text{supp}}

\textheight=1.02\textheight  

\allowdisplaybreaks
\title[Maximal regularity for non-autonomous evolution equations]{Maximal regularity for  non-autonomous evolution equations  governed by forms  having  less regularity }
\address{%
Institut de Math\'ematiques de Bordeaux, CNRS UMR 5251 \\Univ.  Bordeaux \\351, cours de
la Lib\'eration\\33405 Talence CEDEX\\FRANCE}
\email{
Elmaati.Ouhabaz@math.u-bordeaux.fr}
\author{El Maati Ouhabaz}
\thanks{The research of the author was partially supported by the ANR
project HAB, ANR-12-BS01-0013-02}

\begin{abstract}
 We consider the maximal regularity problem for non-autonomous evolution equations
 \begin{equation}
 \left\{
  \begin{array}{rcl}
     u'(t) + A(t)\,u(t) &=& f(t), \ t \in (0, \tau] \\
     u(0)&=&u_0.
  \end{array}
\right.
\end{equation}
Each operator $A(t)$ is associated with a sesquilinear form $\fra(t)$
on a Hilbert space $H$.  We assume that these forms all have the same
domain $V$.  It is proved in  \cite{HO14} that if the forms have some regularity 
with respect to $t$ 
(e.g., piecewise $\alpha$-H\"older continuous for some $\alpha >
\einhalb$) then  the above problem has  maximal $L_p$--regularity for all $u_0 $ in the
real-interpolation space $(H, \DOMAIN(A(0)))_{1-\nicefrac{1}{p},p}$. In this paper  we prove that the  regularity required there can be improved for a class of 
sesquilinear forms. The forms considered here are   such that the difference $\fra(t;\cdot,\cdot) - \fra(s; \cdot,\cdot)$ is continuous on a larger space than the common domain $V$. We give three examples which illustrate our results.
\end{abstract}

\maketitle


\section{Introduction and main results}
Let $H$ and $V$ be  real or complex Hilbert spaces such that  $V$ is 
densely  and continuously  embedded in $H$.   We
denote by $V'$ the (anti-)dual of $V$ and by $\sprod{\cdot}{\cdot}$
the scalar product of $H$ and $\dprod{\cdot}{\cdot}$ the duality
pairing $V'\times V$.  The latter  satisfies (as usual)
$\dprod{v}{h} = \sprod{v}{h}$ whenever $v \in H$ and $h \in V$. By the
standard identification of $H$ with $H'$ we then obtain continuous and
dense embeddings $V \embeds H \eqsim H' \embeds V'$. We denote by
$\norm{.}_V$ and $\norm{.}_H$ the norms of $V$ and $H$, respectively.

\smallskip

We consider  the non-autonomous evolution equation
\begin{equation}\label{eq:evol-eq} \tag{P}
\left\{
  \begin{array}{rcl}
     u'(t) + A(t)\,u(t) &=& f(t), \ t \in (0, \tau] \\
     u(0)&=&u_0,
  \end{array}
\right.
\end{equation}
where each operator $A(t)$, $t \in [0, \tau]$,  is associated with a
sesquilinear form $ \fra(t)$. We  assume that  $t \mapsto  \fra(t;u,v)$ is measurable for all $u, v \in V$ and 
     \let\ALTLABELENUMI\labelenumi \let\ALTTHEENUMI\theenumi
     \renewcommand{\labelenumi}{[H\arabic{enumi}{]}}
     \renewcommand{\theenumi}{[H\arabic{enumi}{]}}
\begin{enumerate}
\item \label{item:constant-form-domain}
  (constant form domain) $\DOMAIN(\fra(t)) = V$.
\item \label{item:uniform-continuity} (uniform boundedness) there
  exists $M>0$ such that for all $t \in [0, \tau]$ and $u, v \in V$,
  we have $|\fra(t; u,v)|\le M \norm{u}_V \norm{v}_V$.
\item \label{item:uniform-accretivity} (uniform quasi-coercivity)
  there exist $\al_1>0$, $\delta \in \RR$ such that for all $t \in [0,
  \tau]$ and all $u, v \in V$ we have $\al_1 \norm{u}_V^2 \le \Re
  \fra(t; u,u) + \delta \norm{u}_H^2$.
\end{enumerate}
     \let\labelenumi\ALTLABELENUMI
     \let\theenumi\ALTTHEENUMI

For each $ t$, we can associate with the form $\fra(t; \cdot, \cdot)$ an operator $A(t)$ defined as follows
\begin{align*}
\DOMAIN(A(t)) &= \{ u \in V, \exists  v \in H:  \fra(t,u,\varphi) = \sprod{v}{\varphi} \, \forall \varphi \in V \}\\
A(t) u &:=  v.
\end{align*}
On the other hand, there exists a  linear operator $\calA(t): V \to V'$
     such that $\fra(t; u,v) = \dprod{\calA(t)u}{v} $ for all $u, v
     \in V$. The operator $\calA(t)$ can be
seen as an unbounded operator on $V'$ with domain $V$ and $A(t)$ is  the part of $\calA(t)$ on
$H$, that is, 
\[
   \DOMAIN(A(t)) = \{ u \in V,\  \calA(t) u \in H \}, \qquad A(t) u = \calA(t) u.
\]
It is a known fact that $-A(t)$ and $-\calA(t)$ both generate
holomorphic semigroups $(e^{-s \, A(t)})_{s \ge 0}$ and $(e^{-s\, \calA(t)})_{s \ge 0}$ on $H$ and $V'$, respectively. For each
$s \ge 0$, $e^{-s \, A(t)}$ is the restriction of $e^{-s \, \calA(t)}$ to $H$.  For all this, we refer to Ouhabaz
\cite[Chapter~1]{Ouhabaz:book}.

\medskip

The notion of maximal $L_p$--regularity for the above Cauchy problem
is defined as follows.
\begin{definition}\label{def:max-reg}
  Fix $u_0 \in H$.  We say that (\ref{eq:evol-eq}) has maximal
  $L_p$--regularity (in $H$) if for each $f \in L_p(0,\tau; H)$ there
  exists a unique $u \in W^{1}_p(0, \tau; H)$, such that $u(t) \in
  \DOMAIN(A(t))$ for almost all $t$, which satisfies
  (\ref{eq:evol-eq}) in the $L_p$--sense. 
 \end{definition}
\medskip
 
\noindent 

 \noindent Recall that under the assumptions  \ref{item:constant-form-domain}--~\ref{item:uniform-accretivity}, J.L. Lions 
  proved maximal $L_2$--regularity  in $V'$ for all initial data $u_0 \in
 H$, see e.g. \cite{Lions:book-PDE}, \cite[page 112]{Showalter}. This
 means that for every $u_0 \in H$ and $f \in L_2(0,
 \tau; V')$, the equation
\begin{equation}  \label{eq:lions-pb} \tag{P'}
  \left\{
  \begin{array}{rcl}
     u'(t) + \calA(t)\,u(t) &=& f(t) \\
     u(0)&=&u_0
  \end{array}\right.
\end{equation}
has a unique solution $u \in W^{1}_2(0,\tau ; V') \cap L_2(0, \tau
;V)$.  It 
is a remarkable fact that Lions's theorem does not require any
regularity assumption (with respect to $t$) on the sesquilinear forms
apart from measurability. 
Note however that maximal regularity in $H$ differs considerably from 
maximal regularity in $V'$. The fact that the forms have the same domain means that the operators $\calA(t)$ have constant domains in $V'$ and this fact plays an important role in proving maximal regularity. The operators $A(t)$ may have different domains as operators on $H$. The problem of  maximal regularity in $H$ for \eqref{eq:evol-eq} was  stated explicitly by Lions and it  is still open in general.  Some progress has been made in recent years. \\
First, recall that Bardos \cite{Bar71} proved maximal $L_2$-regularity in $H$ with initial data $u_0 \in V$ provided 
$\DOMAIN(A(t)^\einhalb) = V$ as space and topologically and assuming that $t \mapsto  \fra(t;u,v)$ is $C^1$ on $[0, \tau]$. 
His result was extended in Arendt et al. \cite{ArendtDierLaasriOuhabaz}  for Lipschitz forms (with respect to $t$) and allowing a multiplicative perturbation by bounded operators $B(t)$ which are measurable in $t$. The maximal $L_2$-regularity is then proved for the evolution problem associated with $B(t)A(t)$. Ouhabaz and Spina \cite{OuhabazSpina} proved maximal 
$L_p$-regularity on $H$ for all $p \in (1, \infty)$  under the assumption that  $t \mapsto \fra(t;u,v)$ is $\alpha$-H\"older continuous for some $\alpha > \einhalb$. The result in 
\cite{OuhabazSpina}  concerns the problem \eqref{eq:evol-eq} with initial data $u(0) = 0$. 
A simple example was given recently by Dier \cite{Dier14} which  shows that in general the answer to Lions' problem is negative.  The following positive result  was  proved by  Haak and Ouhabaz \cite{HO14}.
\begin{theorem}\label{thm:HO}
  Suppose that the forms $(\fra(t))_{0 \le t \le \tau}$ satisfy the
   hypotheses
  \ref{item:constant-form-domain}--~\ref{item:uniform-accretivity} and
  the regularity condition
  \begin{equation}
    \label{eq:dini-regularity-of-forms}
   \left| \fra(t; u,v) - \fra(s; u,v) \right| \le \omega(|t{-}s|) \, \norm{u}_V \norm{v}_V
  \end{equation}
  where $\om: [0, \tau] \to [0, \infty)$ is a non-decreasing function such that 
\begin{equation}  \label{eq:Dini-3-2-condition}
    \int_0^\tau \tfrac{\omega(t)}{t^{\nicefrac32}}\,\dt < \infty. 
\end{equation}
Then the Cauchy problem (\ref{eq:evol-eq}) with $u_0 = 0$ has
maximal $L_p$--regularity in $H$ for all $p \in (1, \infty)$. If in addition $\om$
satisfies the $p$--Dini condition
\begin{equation}  \label{eq:p-Dini}
     \int_0^\tau \left(\tfrac{\omega(t)}{t} \right)^p \,\dt <  \infty,
\end{equation}
then (\ref{eq:evol-eq}) has maximal $L_p$--regularity for all $u_0 \in (H,
\DOMAIN(A(0)))_{1- \nicefrac1{p}, p}$. Moreover  there exists a positive constant $C$
such that
\begin{align*}
&\norm{u}_{L_p(0, \tau; H)} + \norm{u'}_{L_p(0, \tau; H)} + \norm{ A(\cdot) u(\cdot)}_{L_p(0, \tau; H)}\\
& \le C\left[ \norm{f}_{L_p(0, \tau; H)} + \norm{u_0}_{(H, \DOMAIN(A(0)))_{1-\nicefrac1{p}, p}} \right].
 \end{align*}
\end{theorem}
\medskip
In this theorem, $(H, \DOMAIN(A(0)))_{1-\nicefrac1{p}, p}$ denotes the
classical real-interpolation space, see \cite[Chapter
1.13]{Triebel:interpolation} or \cite[Proposition 6.2]{Lunardi:book}.

In the case where $p= 2$, we obtain maximal $L_2$--regularity for $u(0) \in \DOMAIN((\delta + A(0))^\einhalb)$. The theorem can be used in the case where $t \mapsto \fra(t;u,v)$ is $\alpha$-H\"older continuous for some $\alpha > \frac{1}{2}$. The case of piecewise $\alpha$-H\"older continuous is also covered. See \cite{HO14} for the details. 

The aim of the present paper is to weaken the regularity assumption measured by \eqref{eq:Dini-3-2-condition} and \eqref{eq:p-Dini}  in some situations. More precisely, we assume in addition to  \ref{item:constant-form-domain}--~\ref{item:uniform-accretivity} that 
there exist $\beta, \gamma \in [0,1]$ such that for all $u, v \in V$ 
\begin{equation}\label{interpolation}
 \left| \fra(t; u,v) - \fra(s; u,v) \right| \le \omega(|t{-}s|) \, \norm{u}_{V_\beta} \norm{v}_{V_\gamma},
 \end{equation}
 where $V_\beta := [H, V]_\beta$ is the classical complex interpolation space for $\beta \in [0, 1]$ with $V_0 = H$ and 
 $V_1 = V$.   If  $\beta, \gamma \in (0,1)$, the assumption \eqref{interpolation} means that the difference of the forms is defined on a larger space than the common form domain $V$. \\
Our main result is the following.
\begin{theorem}\label{thm:O}
  Suppose that the forms $(\fra(t))_{0 \le t \le \tau}$ satisfy the
  hypotheses
  \ref{item:constant-form-domain}--~\ref{item:uniform-accretivity} and \eqref{interpolation}
  where $\om: [0, \tau] \to [0, \infty)$ is a non-decreasing function such that 
\begin{equation}  \label{eq:Dini-condition}
    \int_0^\tau \tfrac{\omega(t)}{t^{1+ \frac{\gamma}{2}}}\,\dt < \infty. 
\end{equation}
Then the Cauchy problem (\ref{eq:evol-eq}) with $u_0 = 0$ has
maximal $L_p$--regularity in $H$ for all $p \in (1, \infty)$. If in addition $\om$
satisfies the $p$--Dini condition
\begin{equation}  \label{Dini} 
     \int_0^\tau \left(\tfrac{\omega(t)}{t^{\tfrac{\beta + \gamma}{2}}}  \right)^p \,\dt <  \infty,
\end{equation}
then (\ref{eq:evol-eq}) has maximal $L_p$--regularity for all $u_0 \in (H,
\DOMAIN(A(0)))_{1- \nicefrac1{p}, p}$. Moreover  there exists a positive constant $C$
such that
\begin{align*}
 &\norm{u}_{L_p(0, \tau; H)} + \norm{u'}_{L_p(0, \tau; H)} + \norm{ A(\cdot) u(\cdot)}_{L_p(0, \tau; H)} \\
&\le C\left[ \norm{f}_{L_p(0, \tau; H)} + \norm{u_0}_{(H, \DOMAIN(A(0)))_{1-\nicefrac1{p}, p}} \right].
\end{align*}
\end{theorem}
\medskip
A related result was  proved recently by Arendt and Monniaux \cite{AM14} who prove maximal $L_2$-regularity under the additional assumptions that the Kato square root property $V = \DOMAIN(A(0)^\einhalb)$ holds, $\beta = \gamma$ in \eqref{interpolation} and an additional growth condition  $\omega(t) \le C t^{\tfrac{\gamma}{2}}$. 
We observe that in our result  $\beta$ does not come into play if $u_0 = 0$. We expect the theorem to be true with $\min(\beta, \gamma)$
in place of $\gamma$ in \eqref{eq:Dini-condition}. 

The following two corollaries follow immediately from the theorem.
\begin{corollary}\label{Cor1}
Suppose that the forms $(\fra(t))_{0 \le t \le \tau}$ satisfy the
  hypotheses
  \ref{item:constant-form-domain}--~\ref{item:uniform-accretivity} and $\alpha$-H\"older continuous in the sense that
  \begin{equation}\label{Holder} 
  | \fra(t,u,v) - \fra(s,u,v) | \le C | t - s |^\alpha \norm{u}_{V_\beta}\norm{v}_{V_\gamma}. 
  \end{equation}
  Then the Cauchy problem (\ref{eq:evol-eq}) with $u_0 = 0$ has
maximal $L_p$--regularity in $H$ for all $p \in (1, \infty)$ provided $\alpha > \tfrac{\gamma}{2}$. 
 If in addition $\al > \tfrac{\beta + \gamma}{2} - \tfrac{1}{p}$,
then (\ref{eq:evol-eq}) has maximal $L_p$--regularity for all $u_0 \in (H,
\DOMAIN(A(0)))_{1- \nicefrac1{p}, p}$. Moreover  there exists a positive constant $C$
such that
\begin{align*}
&\norm{u}_{L_p(0, \tau; H)} + \norm{u'}_{L_p(0, \tau; H)} + \norm{ A(\cdot) u(\cdot)}_{L_p(0, \tau; H)}\\
 &\le C\left[ \norm{f}_{L_p(0, \tau; H)} + \norm{u_0}_{(H, \DOMAIN(A(0)))_{1-\nicefrac1{p}, p}} \right].
\end{align*}
\end{corollary}

\begin{corollary}\label{Cor2}
Suppose that the forms $(\fra(t))_{0 \le t \le \tau}$ satisfy the
 hypotheses
  \ref{item:constant-form-domain}--~\ref{item:uniform-accretivity} and $\alpha$-H\"older continuous in the sense that
  \begin{equation}\label{Holder} 
  | \fra(t,u,v) - \fra(s,u,v) | \le C | t - s |^\alpha \norm{u}_{V_\beta}\norm{v}_{V_\gamma}, 
  \end{equation}
 for some $\alpha > \tfrac{\gamma}{2}$.  Then the Cauchy problem (\ref{eq:evol-eq})  has
maximal $L_2$--regularity in $H$  for all $u_0 \in \DOMAIN((\delta + A(0))^\einhalb)$. 
Moreover  there exists a positive constant $C$
such that
\begin{align*} 
&\norm{u}_{L_p(0, \tau; H)} + \norm{u'}_{L_p(0, \tau; H)} + \norm{ A(\cdot) u(\cdot)}_{L_p(0, \tau; H)}\\
& \le C\left[ \norm{f}_{L_p(0, \tau; H)} + 
\norm{ (\delta + A(0))^\einhalb u_0}_H \right].
\end{align*}
\end{corollary}

\medskip
{\it Notation: } We shall often use  $C$ or $C'$ to denote all inessential constants. We use $W^1_p (0, \tau;E)$ as well as $H^s (\Omega) := W^s_2(\Omega)$ for the classical Sobolev spaces. The first one  is the Sobolev space of order one of $L_p$-functions  on  $(0, \tau)$ with values in a Banach space $E$  and the second one is the Sobolev space of order $s$ of $L_2$ scalar-valued functions acting on a domain $\Omega$.

\section{Proof of the main result}\label{Section2}
Throughout this section we adopt the notation of the introduction. We shall use the strategy and ideas of proof of Theorem \ref{thm:HO} in \cite{HO14}
with some modifications in order to incorporate the additional assumption \eqref{interpolation}.

Recall that the solution $u$ to \eqref {eq:evol-eq} exists in $V'$ by Lions' theorem mentioned in the introduction. The aim is to prove that $u(t) \in \DOMAIN(A(t))$ for almost all 
$t \in [0, \tau]$ and $A(.) u(.) \in L_p(0, \tau ; H)$. From this and the Cauchy problem \eqref{eq:evol-eq} it follows that 
$u \in W^1_p(0,\tau;H)$. 

From now on we  assume without loss of generality that the forms are coercive, that is 
\ref{item:uniform-accretivity} holds with $\delta = 0$. The reason is that by replacing $A(t)$ by $A(t) + \delta$, the solution $v$ 
of  \eqref {eq:evol-eq} is 
$v(t) = u(t) e^{-\delta t}$ and it is clear that $u \in W^1_p(0, \tau; H)$ {\it if and only if} $v \in W^1_p(0, \tau; H)$.

First we have the representation formula (see \cite{HO14} for all what follows)
 \begin{equation}    \label{eq:AT00}
    \begin{split}
              u(t) = & \;  \int_0^t  e^{-(t{-}s)  \calA(t)} (\calA(t){-}\calA(s)) u(s)\,\ds \\ 
                     & \quad + \int_0^t e^{-(t{-}s) \calA(t) } f(s)\,\ds  + e^{-t \,\calA(t)}u_0.
    \end{split}
  \end{equation}
In addition,
 \begin{equation}\label{egal}
 \calA(t) u(t) = (Q \calA(\cdot)u(\cdot))(t) + (L f)(t) + (Ru_0)(t),
 \end{equation} 
where
\begin{align*}
(Q g)(t) := & \; \int_0^t  \calA(t) e^{-(t{-}s) \calA(t)} (\calA(t) - \calA(s)) \, \calA(s)^{-1}g(s)\,\ds\\
(L g )(t) :=& \;  \calA(t) \int_0^t  e^{-(t{-}s) \calA(t)} g(s)\,\ds 
  \quad \text{and} \quad  
  (R u_0)(t) := \calA(t) e^{-t \, \calA(t)} u_0.
\end{align*}
The aim is to prove boundedness on $L_p(0,\tau;H)$ of the  operators  $L$, $R$ and $Q$ and then by a simple scaling argument 
the norm of $Q$ is less than $1$. This allows us to invert $(I-Q)$ on $L_p(0,\tau ; H)$ and conclude from \eqref{egal} that 
$A(.) u(.) \in L_p(0, \tau ; H)$. 

We start with the operator $L$. The following result is Lemma 2.6 in \cite{HO14}.

\begin{lemma}\label{lem1}
Suppose that in addition to the  assumptions
  \ref{item:constant-form-domain}-~\ref{item:uniform-accretivity} that 
  \eqref{interpolation}  holds for some $\beta, \gamma \in [0, 1]$ 
  and  $\om: [0, \tau] \to [0, \infty)$  a non-decreasing function such that 
\begin{equation}  \label{eq:Dini-2-1-condition}
    \int_0^\tau \tfrac{\omega(t)^2}{t}\,\dt < \infty. 
\end{equation}
 Then $L$ is a bounded operator on $L_p(0, \tau; H)$ for all $p \in (1, \infty)$.  
\end{lemma}

Now we deal with the operator $R$. 

Recall first that $-A(t)$ is the  generator of a bounded holomorphic semigoup of angle 
$\frac{\pi}{2} - \arctan(\frac{M}{\alpha_0})$ where $\alpha_0$ and $M$ are as in the assumptions \ref{item:uniform-continuity} and  \ref{item:uniform-accretivity}.
See \cite[Chapter~1]{Ouhabaz:book} or \cite{HO14}. In addition we have
\begin{lemma}\label{lem2}
Let $\omega: \RR \to \RR_+$ be some function  and 
assume that
\[
| \fra(t; u,v)-\fra(s; u,v)| \le \omega(| t{-}s|) \norm{u}_{V_\beta} \norm{v}_{V_\gamma}
\]
for all $u, v  \in V$. Then 
\[
   \norm{ R(z,A(t)) - R(z,A(s)) }_{\BOUNDED(H)} \le \tfrac{ c_\theta}{|z|^{1 - \frac{\beta + \gamma}{2}} }\omega(| t{-}s|)
\]
for all $z \notin \sector{\theta}$ with any fixed  $\theta > \arctan(\nicefrac{M}{\alpha})$. The constant $c_\theta$ is independent of $z, t$ and $s$. 
\end{lemma}

\begin{proof}
Fix $\theta > \arctan(\nicefrac{M}{\alpha})$. Note  that (see \cite{HO14}, Proposition 2.1 d))
\begin{equation}\label{resol}
      \norm{ ( z - A(t) )^{-1} x}_V \le \tfrac{C_\theta}{\sqrt{|z|}} \norm{x}_H \ \text{for\ all }\  z \notin \sector{\theta}.
         \end{equation}
Observe that for $u, v \in V$,
\begin{align*}
    & \;\bigl|\sprod{ R(z,A(t)) u - R(z,A(s)) u}{v}\bigr| \\
  = &\; \bigl|\sprod{ R(z,A(t)) ( A(s)-A(t) ) R(z,A(s)) u}{v}\bigr| \\
  = &\; \bigl|\sprod{ A(s) R(z,A(s)) u} { R(z,A(t))^* v} - \sprod{ A(t)R(z,A(s))u} { R(z,A(t))^* v}\bigr| \\
  = &\; \bigl|\fra(s; R(z,A(s)) u, R(z,A(t))^* v) - \fra(t; R(z,A(s)) u, R(z,A(t))^* v)\bigr|\\
  \le&\;  \omega(|t-s|)  \norm{R(z,A(s)) u}_{V_\beta} \norm{R(z,A(t))^* v}_{V_\gamma}\\
\le & \; \tfrac{ c_\theta}{|z|^{2 - \frac{\beta + \gamma}{2}} }\omega(| t{-}s|) \, \norm{u}_H\,\norm{v}_H,
\end{align*}
where we used  \eqref{resol} and a standard interpolation argument.
  \end{proof}

\begin{lemma}\label{lem3} 
  Assume (\ref{Dini}). Then there exists $C > 0$ such that 
  \[ 
\norm{Ru_0}_{L_p(0, \tau;H)} \le  C \norm{u_0}_{( H, \DOMAIN(A(0)))_{1-\nicefrac{1}{p},p}},
\]
for all  $u_0 \in ( H, \DOMAIN(A(0)))_{1-\nicefrac{1}{p},p}$. 
\end{lemma}
\begin{proof}
  Recall that the operator $R$ is given by $(R g)(t) = A(t) e^{-t \,A(t)} g$ for $g \in H$.  Let
\[
   (R_0 g)(t) := A(0) e^{-t \,  A(0)}g.
\]
We  estimate the difference $(R-R_0)g$.  Let $v \in H$ and $\Gamma = \partial
\sector{\theta}$ with $\theta < \pihalbe$ as in \eqref{resol}.
 Then the functional calculus for the sectorial operators $A(t)$ and
$A(0)$ gives
\begin{align*}
   & \;  \sprod{ A(t)  e^{-t \,  A(t)} g -  A(0)  e^{-t \,  A(0)}g }{v}\\
 = & \;  \tfrac1{2\upi  i} \int_\Ga  \sprod{ z e^{-t z} \bigl[ R(z, A(t)) - R(z, A(0))\bigr] g }{v} \, \dz\\
 = & \;  \tfrac1{2\upi  i} \int_\Ga  \sprod{ z e^{-t z} R(z, \calA(t)) \bigl[\calA(0)  - \calA(t)\bigr] R(z, A(0))g }{v} \, \dz\\
 = & \;  \tfrac1{2\upi  i} \int_\Ga  \sprod{ z e^{-t z} \bigl[\calA(0)  - \calA(t)\bigr] R(z, A(0))g }{R(z, A(t))^*v} \, \dz\\
 = & \;  \tfrac1{2\upi  i} \int_\Ga z e^{-t z} \bigl[\fra(0;  R(z, A(0))g, R(z, A(t))^* v) - \\
 & \hspace{3cm} \fra(t; R(z, A(0))g, R(z, A(t))^*v )\bigr] \,\dz.
\end{align*}
It follows from  \eqref{interpolation} and Lemma \ref{lem2} that 
\begin{align*}
& \left|\sprod{ (Rg - R_0 g)(t)}{v} \right|\\
\le &\;  \tfrac{1}{2\upi } \int_\Ga \omega(t) |z| e^{-t \,\Re(z)} \norm{ R(z, A(0))g }_{V_\beta} \norm{ R(z, A(t))^*v }_{V_\gamma}  \,|\dz|\\
\le &\;  C  \omega(t) \norm{g}_H \norm{v}_H  \int_\Ga   |z|^{\frac{\beta + \gamma}{2} -1}   e^{-t \,\Re z} \,|\dz| \\
\le&\,   C'  \tfrac{\omega(t)}{t^{\frac{\beta + \gamma}{2}}}  \norm{g}_H \norm{v}_H.
\end{align*}
Since this true for all $v \in H$ we conclude that 
\begin{equation}\label{RR}
\norm{(R u_0)(t)  - (R_0 u_0)(t)}_H \le C'  \tfrac{\omega(t)}{t^{\frac{\beta + \gamma}{2}}}  \norm{u_0}_H.
\end{equation}
From the hypothesis (\ref{Dini}) it follows that $Ru_0- R_0 u_0
\in L_p(0, \tau; H)$.  On the other hand, since $A(0)$ is invertible,
it is well-known that $A(0) e^{-t \,A(0)}u_0 \in L_p(0, \tau; H)$ if and
only if $u_0 \in ( H, \DOMAIN(A(0)))_{1-\nicefrac{1}{p},p}$ (see Triebel
\cite[Theorem~1.14]{Triebel:interpolation}). Therefore, $Ru_0 \in
L_p(0, \tau;H)$ and the lemma is proved.
  \end{proof}

\begin{proof}[Proof of Theorem~\ref{thm:O}]
  As we already mentioned before, the arguments are essentially the same as in  \cite{HO14} in which we use the additional assumption \eqref{interpolation} to weaken the required regularity on the forms.
  We start with the case  $u_0 = 0$ and let $f \in C_c^\infty(0, \tau; H)$. From \eqref{egal}
 we have 
  \begin{equation}    \label{eq:intermediate}
   (I-Q) A(\cdot)u(\cdot) = L f(\cdot).
  \end{equation}
Recall that $L$ is bounded on $L_p(0, \tau; H)$ by Lemma \ref{lem1}. We shall now prove that $Q$ is bounded on $L_p(0, \tau; H)$. 
Let $g \in L_2(0, \tau; H)$ and $v \in H$. 
We have 
\begin{align}
& | \sprod{Qg(t)}{v} |  \nonumber\\
=  &\; \int_0^t \bigl[ \fra(t; \calA(s)^{-1}g(s), \calA(t)^* e^{-(t{-}s)  \calA(t)^*)}v ) - \\
& \hspace{3cm} \fra(s; \calA(s)^{-1}g(s), \calA(t)^* e^{-(t{-}s)  \calA(t)^*)} v)\bigr]  \,\ds \nonumber \\
\le  &\; \int_0^t  \omega(|t-s|) \norm{\calA(s)^{-1}g(s)}_{V_\beta} \norm{ \calA(t)^* e^{-(t{-}s)  \calA(t)^*)}v}_{V_\gamma} \,\ds.\label{truc}
\end{align}
By coercivity assumption one has easily
\[
\norm{\calA(t) e^{-s \calA(t)} v}_V \le \tfrac{C}{s^{\frac{3}{2}}} \norm{v}_H
\]
(see Proposition 2.1 c) in \cite{HO14}). Hence by interpolation
\begin{equation}\label{sgVgamma}
\norm{\calA(t) e^{-s \calA(t)} v}_{V_\gamma} \le \tfrac{C}{s^{ 1 + \frac{\gamma}{2}}} \norm{v}_H.
\end{equation}
The constant $C$ is independent of $t, s$ and $v$. The adjoint operators $\calA(t)^*$ satisfy the same estimates.

Now we estimate $\norm{ \calA(s)^{-1} g(s)}_{V_\beta}$. By coercivity
\begin{align*}
 \norm{ \calA(s)^{-1} g(s)}_{V_\beta}^2 \le C&\;  \norm{ \calA(s)^{-1} g(s)}_{V}^2\\
 \le  &\;  \tfrac{C}{\alpha_0} \Re  \fra(s ; \calA(s)^{-1}g(s), \calA(s)^{-1}g(s))\\
=&\;   \tfrac{C}{\alpha_0} \Re \langle \calA(s) \calA(s)^{-1}g(s), \calA(s)^{-1}g(s) \rangle\\
=&\;  \tfrac{C}{\alpha_0} \Re \sprod{g(s)}{\calA(s)^{-1}g(s)}\\
\le&\; \tfrac{C}{\alpha_0} \norm{g(s)}_H^2 \norm{\calA(s)^{-1}}_{\BOUNDED(H)}.
\end{align*}
Inserting this and \eqref{sgVgamma} (for the adjoint operators)  in  \eqref{truc} we obtain
\begin{equation}\label{est00}
\norm{ (Qg) (t)}_H    \le  \int_0^t \tfrac{C'}{(t{-}s)^{1+ \nicefrac{\gamma}{2}}} \, \omega(t{-}s) \,   \norm{\calA(s)^{-1}}_{\BOUNDED(H)}^{\einhalb}  \norm{g(s)}_H \,\ds. 
\end{equation}
Now, once we replace $A(s)$ by $A(s){+}\mu$, \eqref{sgVgamma} is valid
with a constant independent of $\mu \ge 0$  and using the
estimate
\[
 \norm{(\calA(s) + \mu)^{-1}}_{\BOUNDED(H)} \le \tfrac{1}{\mu},
 \]
 in \eqref{est00} for  $A(s){+}\mu$ we see that
\[
   \norm{ (Qg) (t)}_H \le    \tfrac{C'}{\sqrt{\mu}}  \int_0^t \tfrac{\omega(t{-}s)}{(t{-}s)^{1 + \nicefrac{\gamma}{2}}}  \,   \norm{g(s)}_H \,\ds.
\]
The operator $S$ defined  by 
 \[
   Sh(t) := \int_0^t \tfrac{\omega(t{-}s)}{(t{-}s)^{1 + \nicefrac{\gamma}{2}}}  h(s) \, \ds
\]
is bounded on $L_p(0,\tau; \RR)$ since it has a kernel  
${\omega(t{-}s)}{(t{-}s)^{1 + \nicefrac{\gamma}{2}}}$ which integrable with respect to each variable uniformly with respect to the other variable by \eqref{eq:Dini-condition}. 
It follows that $Q$ is bounded on $L_p(0, \tau; H)$ with
norm of at most $\tfrac{C''}{\sqrt{\mu}}$ for some constant
$C''$. Taking then $\mu $ large enough makes $Q$ strictly contractive
such that $(I-Q)^{-1}$ is bounded on $L_p(0,\tau;H)$.  Then, for $f
\in C_c^\infty(0, \tau; H)$, (\ref{eq:intermediate}) can be rewritten
as
\[
    A(\cdot)u(\cdot) =  (I-Q)^{-1} L f(\cdot).
\]
This shows that $u(t) \in \DOMAIN(A(t))$ for almost $t$ and $A(\cdot)u(\cdot) \in L_p(0, \tau; H)$.

For  general $u_0 \in (H, \DOMAIN(A(0)))_{1-\nicefrac1{p}, p}$ we suppose in
addition to \eqref{eq:Dini-condition} that \eqref{Dini}
holds. Lemma~\ref{lem3} shows that $R u_0 \in L_p(0, \tau;
H)$.  As previously we conclude that 
\[
   A(\cdot)u(\cdot) = (I-Q)^{-1} (Lf + R u_0),
\]
whenever $f \in C_c^\infty(0, \tau; H)$. Thus taking the $L_p$ norm yields
\[ 
   \norm{A(\cdot)u(\cdot)}_{L_p(0, \tau;H)} \le C \norm{(Lf + R u_0)}_{L_p(0, \tau;H)}.
\]
We use again the previous estimates on $L$ and $R$ to obtain
\[
    \norm{ A(\cdot) u(\cdot)}_{L_p(0, \tau; H)} \le C' \left[ \norm{f}_{L_p(0, \tau; H)} + \norm{u_0}_{(H, \DOMAIN(A(0)))_{1-\nicefrac1{p}, p}} \right].
\]
Using the equation (\ref{eq:evol-eq}) we obtain a similar  estimate  for $u'$ and so
\begin{align*}
   & \norm{ u'(\cdot)}_{L_p(0, \tau; H)} +  \norm{ A(\cdot) u(\cdot)}_{L_p(0, \tau; H)} \\
   &\le C'' \left[ \norm{f}_{L_p(0, \tau; H)} + \norm{u_0}_{(H, \DOMAIN(A(0)))_{1-\nicefrac1{p}, p}} \right].
\end{align*}
 We write $u(t) = A(t)^{-1} A(t) u(t)$ and use one again the fact that  the norms of $A(t)^{-1}$ on $H$ are uniformly bounded we obtain
\begin{align*}
  & \norm{ u(t) }_{L_p(0, \tau; H)} \le C_1 \norm{ A(\cdot) u(\cdot)}_{L_p(0, \tau; H)}\\
  & \le C_2 \left[ \norm{f}_{L_p(0, \tau; H)} + \norm{u_0}_{(H, \DOMAIN(A(0)))_{1-\nicefrac1{p}, p}} \right].
\end{align*}
 We conclude therefore that the following a priori estimate holds
\begin{align}  \label{eq:etoile}
  & \norm{ u }_{L_p(0, \tau; H)} +  \norm{ u' }_{L_p(0, \tau; H)} +  \norm{ A(\cdot)u(\cdot)  }_{L_p(0, \tau; H)} \nonumber\\
    &\; \le \; C \left[ \norm{f}_{L_p(0, \tau; H)} + \norm{u_0}_{(H, \DOMAIN(A(0)))_{1-\nicefrac1{p}, p}} \right],
\end{align}
where the constant $C$ does not depend on $f \in C_c^\infty(0, \tau; H)$.

 The latter estimate extends by density to all $f \in L_p(0, \tau;H)$ (see \cite{HO14}). This proves the desired maximal $L_p$-regularity
 property.
 
 


\end{proof}

\section{Examples}

\noindent {\it Schr\"odinger  operators with time dependent potentials.}

\vspace{.2cm}
We consider on $H = L^2(\RR^d)$ Schr\"odinger operators $A(t) = -\Delta + m(t,.)$ with time dependent potentials $m(t,x)$. 
We make the  following assumptions:\\
- There exists a non-negative function $m_0 \in L_{1,loc}$ and two positive constants $c_1, c_2$ such that 
\begin{equation}\label{hypCont}
c_1 m_0(x) \le m(t,x) \le c_2 m_0(x),  \, \, x \in \RR^d, \, t \in [0, \tau].
\end{equation}
- There exists  a function $p_0 \in L_{1,loc}$  such that 
\begin{equation}\label{hypHold}
| m(t,x) - m(s,x) | \le |t-s|^\alpha p_0(x),  \,  x \in \RR^d, \, t, s \in [0, \tau].
\end{equation}
- There exists $C > 0$ and $s \in [0, 1]$  such that 
\begin{equation}\label{hypSob}
\int_{\RR^d} p_0(x) | u(x) |^2\,  \dx \le C \norm{u}_{H^s(\RR^d)},  \, u \in C_c^\infty.
\end{equation}

Note that assumption \eqref{hypSob} is satisfied for several weights $p_0$. For example, this is the case for $p_0 = \frac{1}{|x|^2}$ and $s= 1$ by Hardy's inequality. On the other hand, by H\"older's inequality and classical Sobolev embeddings for $H^s$ one finds $r_s$ such that \eqref{hypSob} holds for $p_0 \in L_{r_s}$. Obviously, 
\eqref{hypSob} holds with $s = 0$ if $p_0 \in L_\infty$. 

The operator $A(t) = -\Delta + m(t,x)$ is  defined as the  operator associated with the form
$$\fra(t;u,v) = \int_{\RR^d} \nabla u . \nabla v\, \dx + \int_{\RR^d} m(t,.) u v\, \dx$$
defined on  
$$ V = \{ u \in H^1(\RR^d), \ \int_{\RR^d} m_0 | u |^2\, \dx < \infty \}.$$
The forms $\fra(t; \cdot, \cdot)$ satisfy the standard assumptions  \ref{item:constant-form-domain}--~\ref{item:uniform-accretivity}. Using the additional 
assumption \eqref{hypSob} we can estimate the difference $\fra(t; u, v) - \fra(s; u, v)$ as follows
\begin{align*}
| \fra(t;u,v) - \fra(s;u,v) | &= |  \int_{\RR^d} [ m(t,.) -m(s,.)]  u v\, \dx |\\
&\le |t-s|^\alpha \int_{\RR^d} p_0(x) |u v|\, \dx\\
&\le  |t-s|^\alpha (\int_{\RR^d} p_0(x) |u|^2\, \dx)^{1/2} (\int_{\RR^d} p_0(x) |v|^2\, \dx)^{1/2}\\
&\le C |t-s|^\alpha \norm{u}_{H^s(\RR^d)} \norm{v}_{H^s(\RR^d)}.
\end{align*}
Therefore, we can apply Theorem \ref{thm:O} to obtain maximal $L_p$-regularity for the  evolution equation associated with
$A(t) = -\Delta + m(t,.)$ under the condition $\alpha > s/2$ where 
$\alpha$ and $s$ are as in \eqref{hypHold} and \eqref{hypSob}. For $p = 2$, the initial data $u_0$ can be taken in $V = \DOMAIN(A(0)^\einhalb)$.
For $p \not=2$ we assume  $u_0 \in (H,\DOMAIN(A(0)))_{1-\nicefrac{1}{p}, p}$ and $\alpha > \max(s/2, s-1/p)$ by condition \eqref{Dini}.\\

\noindent {\it Elliptic operators with Robin boundary conditions.}

\vspace{.2cm}

Let $\Omega$ be a bounded domain of $\RR^d$ with Lipschitz boundary $\partial \Omega$. We denote by ${\rm Tr}$ the classical trace operator. Let  
$\beta: [0, \tau] \times \partial \Omega \to [0, \infty)$ and $a_k : [0, \tau] \times \Omega \to \RR$ be  bounded measurable functions for $k = 1, \cdots,d$ such that 
$$ | \beta(t,x) - \beta(s,x)| \le C |t-s|^\alpha, \, t,s \in [0, \tau], x \in \partial \Omega$$
and 
$$| a_k(t,x) - a_k(s,x)| \le C |t-s|^\alpha, \, t,s \in [0, \tau], x \in \Omega.$$
We define the form
$$\fra(t;u,v) := \int_{\Omega} \nabla u . \nabla v\, \dx +  \sum_{k=1}^d \int_\Omega a_k(t,x) \partial_k u. v\, \dx + \int_{\partial \Omega} \beta(t,.) {\rm Tr} (u) {\rm Tr}(v)\, {\rm d}\sigma,$$
 for all $ u, v \in H^1(\Omega).$
The associated operator $A(t)$ is formally given by
$$A(t) = -\Delta + \sum_{k=1}^d  a_k(t,x) \partial_k u$$
and subject to the time dependent Robin boundary condition:
$$\tfrac{\partial u}{\partial n} + \beta(t,.) u = 0 \ {\rm on }\  \partial \Omega.$$
Here $\tfrac{\partial u}{\partial n}$ denotes the normal derivative.\\
Now we check \eqref{interpolation}. We have for $u, v \in H^1(\Omega)$,
\begin{align*}
& | \fra(t;u,v) -\fra(s;u,v) | \\
&= | \sum_{k=1}^d \int_\Omega [ a_k(t,.) - a_k(s,.)]  \partial_k u. v\, \dx + \int_{\partial \Omega} [\beta(t,.) - \beta(s,.)]  {\rm Tr}(u) {\rm Tr}(v)\, {\rm d}\sigma | \\
&\le  C |t-s|^\alpha \left( \norm{u}_{H^1(\Omega)} + \norm{u}_{H^{\einhalb}(\Omega)} \norm{v}_{H^{\einhalb}(\Omega)} \right),
\end{align*}
where we used the fat that the trace operator is bounded from $H^{\einhalb}(\Omega)$ into $L_2(\partial \Omega)$. Hence 
$$| \fra(t;u,v) -\fra(s;u,v) |  \le C |t-s|^\alpha \norm{u}_{H^1(\Omega)} \norm{v}_{H^{\einhalb}(\Omega)}.$$
We apply Theorem \ref{thm:O} or the subsequent corollaries to  obtain maximal $L_2$-regularity for the corresponding evolution equation 
 under the condition  $\alpha > 1/4$ for  initial data $u(0)  \in H^1(\Omega) = \DOMAIN(A(0)^\einhalb)$. We also have maximal $L_p$-regularity for $1 < p < \infty$ if 
 $\alpha > \max(\frac{1}{4}, \frac{3}{4} - \frac{1}{p})$ and 
 $u (0) \in (H,\DOMAIN(A(0)))_{1-\nicefrac{1}{p}, p}$. 
 In the case $p = 2$  and $a_k = 0$, this result was proved in \cite{AM14}.\\

\noindent{\it Elliptic operators with Wentzell boundary conditions.}

\vspace{.2cm}
We wish to consider the heat equation with time dependent Wentzell boundary conditions:
\begin{equation}\label{wentzell}
\beta(t,.) u   + \frac{\partial u}{\partial n} + \Delta u = 0 \, {\rm on}\ \partial \Omega.
\end{equation}
As in the previous example,  we assume that $\Omega$ is a bounded Lipschitz domain and $\beta: [0, \tau] \times \partial \Omega \to [0, \infty)$
is  a bounded measurable function such that
$$ | \beta(t,x) - \beta(s,x)| \le C |t-s|^\alpha, \, t,s \in [0, \tau], x \in \partial \Omega.$$ 
In order to consider the Laplacian with Wentzell  boundary conditions it is convenient to work on $ H := L_2(\Omega) \oplus L_2(\partial \Omega)$ (see \cite{AMRP} or \cite{FGGR}).  Set
$$V = \{ (u, {\rm Tr}(u)), u \in H^1(\Omega) \}$$
and  define the form
$$ \fra(t;(u, {\rm Tr}(u)) , (v, {\rm Tr}(v)) = \int_\Omega \nabla u.\nabla v\, \dx + 
\int_{\partial \Omega} \beta(t,.) {\rm Tr} (u) {\rm Tr}(v)\, {\rm d}\sigma, $$
for $  u, v \in
 H^1(\Omega).$
The forms $\fra(t)$ are well defined on $V$ and satisfy the assumptions \ref{item:constant-form-domain}--~\ref{item:uniform-accretivity}. 
In addition,
\begin{align*}
& | \fra(t;(u, {\rm Tr}(u)) , (v, {\rm Tr}(v)) - \fra(s;(u, {\rm Tr}(u)) , (v, {\rm Tr}(v))| \\
&\le \int_{\partial \Omega} | \beta(t,.) - \beta(s,.) | |{\rm Tr} (u) {\rm Tr}(v)|\, {\rm d}\sigma\\
&\le C |t-s|^\alpha \norm{{\rm Tr} (u)}_{L_2(\partial \Omega)} \norm{{\rm Tr} (v)}_{L_2(\partial \Omega)} \\
&\le C |t-s|^\alpha \norm{(u, {\rm Tr}(u))}_H \norm{(v, {\rm Tr}(v))}_H.
\end{align*}
We apply again Theorem \ref{thm:O} and obtain maximal $L_p$-regularity on $L_2(\Omega) \oplus L_2(\partial \Omega)$ for all $p \in (1, \infty)$ and $u(0) \in H^1(\Omega)$ under the sole condition that $\alpha > 0$. \\

\medskip{\bf Acknowledgements:}  The author wishes to thank Wolfgang Arendt, Bernhard Haak and Sylvie Monniaux  for  various discussions on the subject of this paper.

\def\SUBMITTED{Submitted}
\def\TOAPPEAR{To appear in }
\def\PREPARATION{In preparation }

\def\cprime{$'$}
\providecommand{\bysame}{\leavevmode\hbox to3em{\hrulefill}\thinspace}

\end{document}